\documentclass[times]{amsart}

\usepackage{amsmath, amsthm, amssymb, amsfonts}
\usepackage{calc}
\usepackage{url}
\usepackage{booktabs}
\usepackage{array}
\usepackage{paralist}
\usepackage{verbatim}
\usepackage{mathcomp}
\usepackage{enumerate}
\usepackage{xfrac}
\usepackage{mathtools}
\usepackage{amscd}
\usepackage{stmaryrd}
\usepackage{float}
\usepackage{cite}
\usepackage{tikz-cd}
\usepackage{graphicx}
\usepackage{adjustbox}
\usepackage{fontenc}
\usepackage{mwe}
\usepackage{caption}
\usepackage{subcaption}
\usepackage{multirow}
\usepackage{lscape}
\usepackage{tkz-graph}
\usepackage{semtkzX}
\usepackage{pbox}
\usepackage{hyperref}

\advance\oddsidemargin -1.5cm
\advance\evensidemargin -1.5cm
\advance\textwidth 2cm
\advance\textheight 0.3cm 

\def\Q{\mathbb{Q}}
\def\Z{\mathbb{Z}}
\def\C{\mathbb{C}}
\def\R{\mathbb{R}}
\def\F{\mathbb{F}}

\def\Gm2{\mathbb{G}_m^2}

\def\Jac{\textup{Jac }\!}

\def\im{\textup{Im}}
\def\prym{\textup{Prym}}

\def\rk{\textup{rk}}
\def\ord{\textup{ord}}
\def\Pic{\textup{Pic}}

\def\coker{\textup{coker}\,}

\input cyracc.def
\font\tencyr=wncyr10
\def\sha{\text{\tencyr\cyracc{Sh}}}

\newcommand*{\MAGMA}{\textsc{magma}\xspace}

\renewcommand{\mod}[1]{\,\,\, (\textup{mod } #1)}
\renewcommand{\epsilon}{\varepsilon}

\numberwithin{equation}{section}
\newtheorem{theorem}{Theorem}[section]
\newtheorem{corollary}[theorem]{Corollary}
\newtheorem{conjecture}[theorem]{Conjecture}
\newtheorem{lemma}[theorem]{Lemma}
\newtheorem{proposition}[theorem]{Proposition}

\theoremstyle{definition}
\newtheorem{definition}[theorem]{Definition}

\newtheorem{remark}[theorem]{Remark}

\begin{document}

\title{2$^\infty$-Selmer Rank Parities via the Prym Construction}
\author{Jordan Docking}

\address{Department of Mathematics, University College London, London WC1H 0AY, UK}
\email{jordan.docking.18@ucl.ac.uk}

\begin{abstract}
We derive a local formula for the parity of the $2^{\infty}$-Selmer rank of Jacobians of curves of genus $2$ or $3$ which admit an unramified double cover. We give an explicit example to show how this local formula gives rank parity predictions against which the $2$-parity conjecture may be tested. Our results yield applications to the parity conjecture for semistable curves of genus $3$.
\end{abstract}


\maketitle

\tableofcontents

\section{Introduction}

Let $A$ be an abelian variety over a number field $K$. The Birch--Swinnerton-Dyer conjecture predicts that the Mordell-Weil rank, $\rk(A/K)$, and the order of vanishing of the $L$-function $L(A/K,s)$ at $s=1$ are equal. It is not yet known in general that $L$ can be extended analytically to the point $s = 1$, however it is expected that it can be extended to all of $\C$. Moreover, $L$ is expected to satisfy a functional equation about $s = 1$. If the Birch--Swinnerton-Dyer conjecture is true then the parity of $\rk(A/K)$ should be distinguished by whether $L$ is symmetric or anti-symmetric. Conjecturally, this is controlled by the \textit{root number} $w_{A/K} = \prod_v w_{A/K_v}$ of $A$, where the product ranges over all places $v$ of $K$, and the terms $w_{A/K_v}$ are \textit{local} root numbers.

\begin{conjecture}[Parity Conjecture]
Let $A/K$ be an abelian variety over a number field. Then
$$
(-1)^{\rk (A/K)} = \prod_{v} w_{A/K_v}.
$$
\end{conjecture}

For $p$ prime, write $\rk_p(A/K) = \rk(A/K) + \delta_p$ for the $p^{\infty}$-Selmer rank, where $\delta_p$ is the multiplicity of $\Q_p / \Z_p$ in $\sha(A/K)$. If $\sha(A/K)$ is finite then $\rk_p(A/K) = \rk(A/K)$ for all $p$.

\begin{conjecture}[$p$-Parity Conjecture]
Let $A/K$ be an abelian variety over a number field, and $p$ prime. Then
$$
(-1)^{\rk_p (A/K)} = \prod_{v} w_{A/K_v}.
$$
\end{conjecture}

We specialise to the $p=2$ case. $2^\infty$-Selmer (and indeed $2$-Selmer) groups are central objects in the study of abelian varieties. The standard $2$-descent procedure computes $2^\infty$-Selmer groups, making them the best method of getting rank bounds. Mazur and Rubin use $2^\infty$-Selmer groups to show that elliptic curves with trivial Mordell--Weil group exist over any number field, along with applications to Hilbert's Tenth Problem \cite{mazurrubin}. Bhargava and Shankar use $2$-Selmer groups to bound the average rank of elliptic curves \cite{bhargavashankar}, whilst similar estimates on the average size of $2$-Selmer groups can be obtained for Jacobians of hyperelliptic curves with a rational Weierstrass point \cite{bhargavagross}. Alexander Smith has shown that $2^\infty$-Selmer groups behave in line with heuristics amongst quadratic twists of elliptic curves, with applications to Goldfeld's conjecture \cite{smith2inftyselmer}. 

The $2$-parity conjecture is known for elliptic curves over $\Q$ and more generally over totally real fields \cite{notesonpc}.  It has also been shown for large classes of abelian surfaces over number fields \cite{dokchitser2020parity}. Assuming the finiteness of the Tate--Shafarevich group, the parity conjecture is known for elliptic curves over number fields (again see \cite{notesonpc}).

We describe a common approach to the $2$-parity conjecture, which comes in two steps. First, one expresses the $2^\infty$-Selmer parity in terms of a \textit{local formula}
$$
(-1)^{\rk_2 A} = \prod_{v} \lambda_{A/K_v},
$$
where the terms $\lambda_{A/K_v}$ depend on the abelian variety only locally. Second, one controls the discrepancy between the local terms $\lambda_{A/K_v}$ and the local root numbers $w_{A/K_v}$ as an error term $\lambda_{A/K_v} = e_v \cdot w_{A/K_v}$, such that $\prod_v e_v = 1$.

Local formulae can be of particular value as they give information about $2^\infty$-Selmer groups of abelian varieties without the unhelpful baggage of global data, in principle making them straightforward to evaluate. It is thus of interest to develop local formulae for certain classes of abelian varieties.

\subsection{Main Result} The present article proves a local formula when $A = \Jac(C)$ is the Jacobian of a curve $C$ of genus $g = 2$ or $3$ with a $K$-rational double cover $\pi \colon D \to C$. In the genus $2$ case, this is implied by the existence of a $K$-rational $2$-torsion point, whilst if $C$ has genus $3$, then a $K$-rational $2$-torsion point along with a $K$-rational point of $C$ gives rise to such a double cover. Such curves have an associated  \textit{Prym variety} $\prym(D/C)$, and isogeny $\phi \colon \Jac C \times \prym(D/C) \to \Jac D$. We illustrate the main theorem as it applies typically.

\begin{theorem}[cf. Theorem \ref{localformula}]
\label{easylocform}
Let $C$ be a smooth projective curve\footnote{Throughout, curves will be assumed to be smooth and projective.} of genus $2$ or $3$ over a number field $K$ with unramified double cover $\pi \colon D \to C$. Suppose moreover that the induced Prym variety is a Jacobian, $\prym(D/C) = \Jac(F)$. Suppose further that $C, D$ and $F$ have points locally for every completion of $K$, and write $\phi_v \colon \Jac C (K_v) \times \Jac F (K_v) \to \Jac D (K_v)$ for the map on local points. Then 
$$
(-1)^{\rk_2 \Jac C + \rk_2 \Jac F} = \prod_v (-1)^{\dim_{\F_2} \left( \frac{\coker \phi_v}{\ker \phi_v} \right) }.
$$
\end{theorem}

In the full theorem, both the condition on the Prym variety being a Jacobian and the condition on local points can be removed. The right-hand side is again a product of terms, written $\lambda_{C/K_v, \phi_v}$ (see Definition \ref{lambdadef}), where once again the data depends on the curve only locally.

We are lead to conjecture the following product formula; along with Theorem \ref{localformula} this gives the parity conjecture for product abelian varieties of the form $\Jac (C) \times \prym(D/C)$.

\begin{conjecture}
\label{loceqroot}
Whenever $C$ is a curve of genus  $2$ or $3$ over a number field $K$ with unramified double cover $\pi \colon D \to C$ and associated Prym variety $\prym(D/C)$ over $K$,
$$
\prod_v \lambda_{C/K_v, \phi_v} \cdot w_{\Jac C/K_v} \cdot w_{\prym(D/C)/K_v} = 1.
$$
\end{conjecture}

Note that $\prod_v \lambda_{C/K_v, \phi_v}$ controls the $2^{\infty}$-Selmer parity of $\Jac C \times \prym(D/C)$, and a priori Conjecture \ref{loceqroot} does not give the Parity Conjecture for $C$ alone. However we can telescope the Prym construction.

\begin{theorem}[= Theorem \ref{reduction}]
Suppose Conjecture \ref{loceqroot} holds. Then the $2$-parity conjecture holds for all curves $C/K$  either of genus $2$ or of genus $3$ with a $K$-rational point, where $K$ is any number field such that $\textup{Gal}(L/K)$ is a $2$-group (and $L = K(\Jac C[2])$).

\end{theorem}

In view of \cite[Thm.~B.1]{dokchitser2020parity} (= Theorem \ref{2group}), this reduces the parity conjecture for a large class of curves of genus $2$ or $3$ to Conjecture \ref{loceqroot}, and the finiteness of the Tate-Shafarevich group.

\begin{corollary}[= Corollary \ref{semistableparity}]
Suppose Conjecture \ref{loceqroot} holds. Let $C/K$ be a semistable curve over a number field $K$, where $C$ is either of genus $2$, or of genus $3$ with a $K$-rational point. If $\sha(\Jac C/L)$ is finite (where $L$ is full the $2$-torsion field of $C/K$), then the parity conjecture holds for $C/K$.
\end{corollary}

\begin{remark}
Local formulae, which express rank parities in terms of local data (as in as in Theorem \ref{easylocform}/\ref{localformula}), have been given elsewhere. Kramer and Tunnell have given one in the setting of quadratic twists \cite[Ch.~3]{kramertunnell}, whilst ideas of a local formula are clear in \cite[App.]{rootnumbersnonab}, where Fisher gives a parity result for elliptic curves with a $p$-isogeny. Dokchitser and Maistret \cite[Thm.~1.8]{dokchitser2020parity} give a local formula  for Jacobians that admit an isogeny $\phi$ satisfying $\phi \phi^t = [2]$. The existence of such isogenies relies on controlling a maximal isotropic subgroup of $\Jac C[2]$. We contrast this directly with our result, which requires a $K$-rational double cover.
\end{remark}

\begin{remark}
This article fulfils the first step (of the two-step approach described above), but we emphasise that no attempt at the second step of proving (or even conjecturing) an error term for our local formula is made here. Indeed, there appear to be two barriers to such a formulation. First, as noted in \cite[Rk.~1.18]{dokchitser2020parity}, where local formulae have been used to prove instances of the parity conjecture there have been no conceptual interpretations of the found error terms. Such an interpretation would likely aid finding them in higher genus cases. Second, these error terms have so far manifested as Hilbert symbols (so that the triviality of their product over all places is immediate). The number of terms involved for the abelian surface case is already substantial \cite[Defn.~1.13]{dokchitser2020parity}, and given the degrees involved with curves of genus $3$ (see, for example, the Dixmier--Ohno invariants \cite{dixmier, ohno}), we suspect that any prospective error term is likely to be unmanageable.
\end{remark}

\subsection{Outline}

This article is organised as follows. We start by recalling the theory of Prym varieties, which are central to our approach. We then develop the local formula for curves $C/K$ of genus $2, 3$ with an unramified double cover. In Section \ref{applications} we show that the parity conjecture for semistable curves of these genera can be reduced to the finiteness of the Tate--Shafarevich group and Conjecture \ref{loceqroot}. In the final two sections we give methods to compute terms in the local formula, and exhibit an explicit example where the $2^{\infty}$-Selmer rank parity is computed.

\subsection{Notation} $K$ will be a number field, $v$ a place of $K$, and $C$ a curve of genus $g$. $\Jac C$ is the Jacobian of $C$, and $\epsilon$ is a (non-trivial) $K$-rational $2$-torsion point of $\Jac C$. 

If $\phi \colon A \to B$ is a $K$-isogeny of abelian varieties, $\phi_v = \phi|_{K_v} \colon A(K_v) \to B(K_v)$ both denote the induced map on local points.

If $K_v$ is a finite extension of $\Q_p$ for some $p$, then $c_{A,v}$ is the Tamagawa number of $A$ at $v$.

A variety $X$ over $\R$ will have $n_{X,\R}$ real components (in the Euclidean topology). $A(\R)^0$ is the real component of $A$ containing the identity. We write $\phi |^0_\R$ for the induced map on the identity component $A(\R)^0 \to B(\R)^0$.

\subsection*{Acknowledgements} I would like to extend my greatest appreciation and thanks to my supervisor Vladimir Dokchitser, for his constant advice and support. Particular acknowledgement must also be given to Holly Green, Omri Faraggi and Raymond van Bommel; without their assistance this work could not have been accomplished, and all have my heartfelt thanks. Furthermore I would like to express my gratitude towards Dominik Bullach and Nils Bruin, conversations with both of whom were of immense help.

This work was supported by the Engineering and Physical Sciences Research Council [EP/L015234/1], the EPSRC Centre for Doctoral Training in Geometry and Number Theory (The London School of Geometry and Number Theory) at University College London.

\section{The Prym Construction}
\label{prymsection}

Let $C$ be a curve over a number field $K$ of genus $g = 2$ or $3$, and, if $g = 3$, suppose that $C(K) \neq \emptyset$. It is well-known that $\epsilon \in \Jac C(K)[2]$ induces an unramified double cover $\pi \colon D \to C$ over $K$, defined up to $K^*/(K^*)^2$, with $g_D = 2g-1$. $D$ inherits a natural involution $\iota \colon D \to D$ which exchanges the sheets above $C$.

We now summarise the standard theory of Prym varieties, with more detailed treatments in \cite{MumfordPryms} and \cite{prymtheory}. Consider the induced map on Jacobians $\pi_{*} \colon \Jac D \to \Jac C$. The \textit{Prym variety} $\prym(D/C)$ of the double cover $\pi$ is the connected component of $\ker \pi_{*}$ (considered as an algebraic group)  containing $0$. We may also write $\prym(\epsilon)$. We will also write $\prym(D/C)$ without introducing $D$ explicitly. Note the equivalent characterisations
$$
\prym(D/C) = \ker (\textup{id} + \iota^*)^0  = \im (\textup{id} - \iota^*),
$$
where the superscript $0$ denotes the Zariski-connected component of the identity. There is a $K$-isogeny (the \textit{Prym isogeny})
$$
\begin{aligned}
\phi \colon \, \Jac C \times \prym(D/C) &\to \Jac D, \\
(x, y) &\mapsto \pi^{*}x + y.
\end{aligned}
$$
The dual isogeny $\phi^t(x) = (\pi_{*}x, x - \iota^{*} x)$ satisfies $\phi^t \circ \phi = [2] = \phi \circ \phi^t$. 

$\prym(D/C)$ is a principally polarized abelian variety of dimension $g-1$, so for $g = 2, 3$ is either a Jacobian, a product of Jacobians, or the Weil restriction of an elliptic curve. A complete description of $\prym(D/C)$ in genus $3$ has been given by Bruin \cite[Thm.~5.1]{BruinGenus3}, and is reproduced in Table \ref{prymtable} (note in Case II.c, $\mathfrak{R}$ denotes the Weil restriction, and in case III.d the $Q_i$ are considered both as conics and as symmetric $3 \times 3$ matrices). 

\begin{table}[h!]
\centering
\resizebox{\textwidth}{!}{%
  \begin{tabular}{|c|c|c|c|c|}
\hline
Case  & Genus of $C$                       & $C$                                                                                      & $D$                                                                                          & $\prym(D/C)$                                                   \\ \hline
II   & 2                                  & $y^2 = f(x)g(x), \, \deg f = 4, \deg g = 2$                                              & $u^2 = f(x), v^2 = g(x)$                                                                     & $\Jac ( y^2 = f(x) ) $                                         \\ \hline
III.a & \multirow{3}{*}{3 (Hyperelliptic)} & $y^2 = f(x)g(x), \, \deg f = 6, \deg g = 2$                                              & $u^2 = f(x), v^2 = g(x)$                                                                     & $\Jac ( y^2 = f(x) ) $                                         \\ \cline{1-1} \cline{3-5} 
III.b &                                    & $y^2 = f(x)g(x), \, \deg f = 4, \deg g = 4$                                              & $u^2 = f(x), v^2 = g(x)$                                                                     & $\Jac (y^2 = f(x) ) \times \Jac (y^2 = g(x) )$                 \\ \cline{1-1} \cline{3-5} 
III.c &                                    & $y^2 = N_{K(\sqrt{d})[x]/K[x]}R(x), \, \deg R = 4$                                       & $(y_0 + y_1 \sqrt{d})^2 = R(x)$                                                              & $\mathfrak{R}_{K(\sqrt{d})/K}\left( \Jac(y^2 = R(x) ) \right)$ \\ \hline
III.d & 3 (Non-hyperelliptic)              & \parbox{6cm}{\centering $Q_1(x,y,z)Q_3(x,y,z) - Q_2(x,y,z)^2 = 0$, \\ $Q_i$ conics} & $\begin{cases} Q_1(u,v,w)  & = r^2 \\ Q_2(u,v,w)  & = rs \\ Q_3(u,v,w)  & = s^2 \end{cases}$ & \parbox{5cm}{\centering $\Jac F,$ \\ $F \colon y^2 = - \det(Q_1 + 2x Q_2 + x^2Q_3$)}                         \\ \hline
\end{tabular}%
}
\caption{Description of Prym varieties in dimension 2 and 3}
\label{prymtable}
\end{table}

\begin{remark}
The Prym construction can also be carried out for elliptic curves. In this case the Prym variety is a point, and one recovers the standard $2$-isogeny for an elliptic curve with a $2$-torsion point. The local formula in Theorem \ref{easylocform} is also seen to agree with \cite[p.~663]{cyclicpc}.
\end{remark}

\section{Derivation of the Local Formula}
\label{derivation}

\begin{definition}
Let $\mathcal{K}$ be a local field. Recall that a curve $X/\mathcal{K}$ of genus $g$ is said to be \textit{deficient} if $\Pic^{g-1}(X_\mathcal{K}) = \emptyset$. We define
$$
\mu_{\mathcal{K},X} = 
\begin{cases}
-1 & \textup{if $X$ is deficient} \\
1 & \textup{otherwise}.
\end{cases}
$$
Note that curves possessing a $\mathcal{K}$-rational point are never deficient. When $K$ is a number field with place $v$, we define $\mu_{v,X} = \mu_{K_v, X}$.
\end{definition}

\subsection{The Local Formula}

In this section we give a precise formulation of the local formula for $2^\infty$-Selmer rank parities.

\begin{definition}
Let $C/K$ be a curve over a number field, with an unramified double cover $\pi \colon D \to C$ and associated Prym variety $\prym(D/C)$. For any place $v$ of $K$, let
$$
\delta_{v,\prym(D/C)} =
\begin{cases}
\mu_{v,F} &\textup{if $\prym(D/C) = \Jac(F)$} \\
\mu_{v,E} \, \mu_{v,E'} &\textup{if $\prym(D/C) = E \times E'$ is a product of elliptic curves}\\
\prod_{w | v} \mu_{w, E} &\textup{if $\prym(D/C) = \mathfrak{R}_{K(\sqrt{d})/K}(E)$}, \, w \in M_{K(\sqrt{d})}.
\end{cases}
$$
\end{definition}

Note that $\delta_{v,\prym(D/C)}$ is always defined when $C$ has genus $2$ or $3$.

\begin{definition}
\label{lambdadef}
Let $C / \mathcal{K}$ be a curve of genus $2$ or $3$ over a local field, with an unramified double cover $\pi \colon D \to C$ and Prym variety $\prym(D/C)$. Let $\phi|_\mathcal{K} \colon \Jac C (\mathcal{K}) \times \prym(D/C) (\mathcal{K}) \to \Jac D (\mathcal{K})$ be the local map on points. Define
$$
\lambda_{C/\mathcal{K}, \phi} = \delta_{v,\prym(D/C)} \frac{\mu_{v,C}}{\mu_{v,D}} \, (-1)^{\dim_{\F_2} \left(  \frac{\coker{\phi |_\mathcal{K}}}{\ker \phi |_\mathcal{K}}\right)}
$$
\end{definition}

\begin{theorem}
\label{localformula}
Let $C / K$ be a curve of genus $2$ or $3$ over a number field, with an unramified double cover $D$ and Prym variety $\prym(D/C)$. Let $\phi \colon \Jac C \times \prym(D/C) \to \Jac D$ be the Prym isogeny. Then
$$
(-1)^{\rk_2 \Jac C + \rk_2 \prym(D/C)} = \prod_v \lambda_{C/K_v, \phi_v}.
$$
\end{theorem}

\begin{proof}
For notational ease write $P = \prym(D/C)$. This proof follows the same method as \cite[Thm.~3.2]{dokchitser2020parity}, the only point of difference is in controlling the size of $\sha^{\textup{nd}}_{(\Jac C \times P) / K} [2^\infty]  \simeq \sha^{\textup{nd}}_{\Jac C/ K}[2^\infty] \times \sha^{\textup{nd}}_{P / K}[2^\infty]$. 

As in \cite[Thm.~4.3]{BSDModSquares},
$$
\begin{aligned}
2^{\rk_2 \Jac C + \rk_2 P} &= \square \cdot  \left( \prod_v \frac{|\coker \phi|_{K_v}|}{|\ker \phi|_{K_v}|} \right) \cdot
\frac{|\sha^{\textup{nd}}_{(\Jac C \times P) / K} [2^\infty] |}{|\sha^{\textup{nd}}_{(\Jac D) / K} [2^\infty] |} \\
\end{aligned}
$$
and hence
\begin{equation}
\label{eqn:2rkparity}
\rk_2 \Jac C + \rk_2 P = \ord_2 \left( \prod_v \frac{|\coker \phi|_{K_v}|}{|\ker \phi|_{K_v}|} \right) + \ord_2 \left(
\frac{|\sha^{\textup{nd}}_{(\Jac C \times P) / K} [2^\infty] |}{|\sha^{\textup{nd}}_{(\Jac D) / K} [2^\infty] |} \right) \mod{2}
\end{equation}

The second term can be understood in terms of $\mu_{v,C}, \mu_{v,D}$, and $\delta_{v,P}$ using the theorem of Poonen and Stoll \cite[Thm.~8, Cor.~12]{poonenandstoll}. Indeed,
$$
|\sha^{\textup{nd}}_{\Jac C/ K}[2^\infty]| = \square \cdot 2^{| \{ v \, | \, \mu_{v,C} \neq 1 \} |}, \quad |\sha^{\textup{nd}}_{\Jac D/ K}[2^\infty]| = \square \cdot 2^{| \{ v \, | \, \mu_{v,D} \neq 1 \} |}
$$
and
$$
| \sha^{\textup{nd}}_{P / K}[2^\infty] | = \square \cdot 2^{| \{ v \, | \, \delta_v \neq 1 \} |}
$$
In particular
$$
\ord_2 \left(  \sha^{\textup{nd}}_{\Jac C / K}[2^\infty] \right)  = | \{ v \, \mid \, \mu_{v,C} \neq 1 \}| \mod 2
$$
Note then that
$$
(-1)^{\ord_2 (  \sha^{\textup{nd}}_{\Jac C / K} [2^\infty] )} = \prod_{v} \mu_{v,C},
$$
and similarly for $P$ and $\Jac D$.
Combining into \ref{eqn:2rkparity}, we have
$$
(-1)^{\rk_2 \Jac C + \rk_2 P} = \prod_v (-1)^{\ord_2 \left(\frac{|\coker \phi|_{K_v}|}{|\ker \phi|_{K_v}|} \right) } \cdot \delta_{v,P}
\frac{\mu_{v,C}}{\mu_{v,D}} = \prod_v \lambda_{C/K_v,\phi_v}$$
as claimed.
\end{proof}

\begin{remark}
The only place where the restriction $g \in \{ 2, 3 \}$ was used was to ensure that $\prym(D/C)$ was amenable to the criterion given in \cite{poonenandstoll}. Interest in extending this result to general Prym varieties was raised immediately by Poonen and Stoll \cite[p.~1126]{poonenandstoll}, and such an extension (perhaps even one covering all principally polarized abelian varieties) would generalize $\delta_{v,\prym(D/C)}$, allowing a local formula in all genera. In turn this would give control over the $2^{\infty}$-Selmer rank parities of general curves (of arbitrary genus) admitting rational double covers.
\end{remark}

\begin{remark}
For curves of genus 2 (Case II of Table \ref{prymtable}), the existence of a rational 2-torsion point (and hence double cover) is equivalent to $\textup{Gal}(fg) \leqslant C_2 \times S_4$. In contrast, \cite[Thm.~1.8]{dokchitser2020parity} gives a local formula in the case that the Galois group is contained in $C_2 \wr S_3$.
\end{remark}

\section{Applications to the Parity Conjecture}
\label{applications}

Here we show how the Prym construction, under mild assumptions, can be iterated to reduce the Parity Conjecture for Jacobians in genus $2$ and $3$ to Conjecture \ref{loceqroot} and the finiteness of the Tate--Shafarevich group.

\begin{lemma}
\label{2tors}
Let $C$ be a curve of genus $g$, with double cover $\pi \colon D \to C$ and Prym variety $\prym(D/C)$. Then $\prym(D/C)[2] \subset \pi^{*}(\Jac C[2])$. In particular, if $C$ has full $2$-torsion over a field $K$, then $\prym(D/C)$ has full $2$-torsion over $K$ also. 
\end{lemma}

\begin{proof}
Recall the characterisations of $\prym(D/C)$ given in Section \ref{prymsection}. Note that $\prym(D/C)[2] \subset \ker (\phi^t)$. Indeed, if $x \in \prym(D/C)[2]$, then $\phi^{t}(x) = (\pi_{*}x,x - \iota^{*}x) = (0,x + \iota^{*}x) = (0,0)$. However we claim that $\ker(\phi^{t}) = \pi^{*}(\Jac C[2])$. First observe that if $\theta \in \Jac C[2]$, then $\phi^{t}(\pi^{*}\theta) = (\pi_{*} \pi^{*} (\theta), \pi^{*}(\theta) - \iota^{*} \pi^{*}(\theta)) = (2\theta, 0) = (0,0)$. Second, as $\ker \pi^{*} = \{0, \epsilon\}$, both sets have size $2^{2g-1}$.
\end{proof}

\begin{theorem}
\label{reduction}
Let $K$ be a number field, and let $C/K$ be either a curve of genus $2$, or a curve of genus $3$ with a $K$-rational point. Suppose $L = K(\Jac C[2])$ is such that $G = \textup{Gal}(L/K)$ is a $2$-group. Then if Conjecture \ref{loceqroot} holds, the $2$-parity conjecture holds for $C/K$.
\end{theorem}

\begin{proof}
We show this when $C$ has genus $3$, the proof in genus $2$ being almost identical. 

Both $G$ and $\Jac C[2]$ are $2$-groups, so the action of $G$ has a non-trivial fixed point, i.e $C$ has a non-trivial $K$-rational $2$-torsion point, $\epsilon$. Let $P = \prym(\epsilon)$ be the corresponding Prym variety. Then $P$ has full $2$-torsion over $K$ by Lemma \ref{2tors}. Thus $\textup{Gal}(K(P[2])/K)$ is a quotient of $G$, hence is also a $2$-group. $P$ then has a $K$-rational $2$-torsion point $\epsilon'$.  We show that the $2$-parity conjecture holds for $\prym(D/C)/K$ by considering each of Cases III.a, III.b, III.c.

In Case III.a, $\prym(D/C) = \Jac F$ is a Jacobian. We can consider the Prym variety of the two-torsion point $\epsilon \in \Jac F$, $\prym(\epsilon') = E$, say. Then $E$ also has full $2$-torsion over $L$, and $\textup{Gal}(K(E[2])/K)$ is a $2$-group, so $E$ has a $2$-torsion point over $K$. The 2-parity conjecture holds for $E$ by \cite[Thm.~5.1]{notesonpc}, i.e. $ (-1)^{\rk_2 E} = \prod_v w_{E/K_v} $. By Conjecture \ref{loceqroot} and Theorem \ref{localformula},
$$
(-1)^{\rk_2 P + \rk_2 E} = \prod_v w_{(P \times E) /K_v} = \prod_v w_{P/K_v} \cdot w_{E/K_v}
$$
Together these give the $2$-parity conjecture for $P/K$. 

In Case III.b, $\prym(D/C) = E \times E'$ is a product of elliptic curves. Both have full 2-torsion over $L$, and so, as above, both have a non-trivial $2$-torsion point in $K$. The $2$-parity conjecture then holds for both over $K$, hence for $\prym(D/C)/K$ also.

In Case III.c, $\prym(D/C) = \mathfrak{R}_{K(\sqrt{d})/K}(E)$ is the Weil restriction of some elliptic curve. By the standard argument, $E$ has a $2$-torsion point over $K(\sqrt{d})$ so the $2$-parity conjecture holds for $E/K(\sqrt{d})$. This is preserved under Weil restriction, i.e. the $2$-parity conjecture holds for $\prym(D/C)/K$.

Conjecture \ref{loceqroot} and Theorem \ref{localformula} give 
$$
(-1)^{\rk_2 \Jac C + rk_2 P} = \prod_v w_{\Jac C/K_v} \cdot w_{P/K_v}
$$
As the $2$-parity conjecture holds for $P/K$, it must then hold for $C/K$.
\end{proof}

\begin{theorem}[{See \cite[Thm.~B.1]{dokchitser2020parity}}]
\label{2group}
Let $L/K$ be a Galois extension of number fields with Galois group $G$ and $A/K$ a principally polarized abelian variety. Suppose
\begin{itemize}
    \item $\sha(A/L)$ has finite $p$-primary part for every odd prime $p$ that divides $|G|$,
    \item $A/K$ is semistable
\end{itemize}
Then if the parity conjecture holds for $A/L^H$ for all $H \leqslant G$ of $2$-power order, it holds for $A/K$.
\end{theorem}

\begin{corollary}
\label{semistableparity}
Suppose Conjecture \ref{loceqroot} holds. Let $C/K$ be a semistable curve of genus $2$ or semistable curve of genus $3$ with a $K$-rational point, over a number field $K$. If $\sha(\Jac C/L)$ is finite (where $L$ is full the $2$-torsion field of $C/K$), then the Parity Conjecture holds for $C/K$.
\end{corollary}

\begin{proof}
By Proposition \ref{reduction}, the $2$-parity conjecture holds for all subfields $L^H \subset L$ with $H$ a $2$-group. By the finiteness of $\sha(\Jac C/L^H)$, the parity conjecture holds for all such $C$, and so by Theorem \ref{2group}, the parity conjecture holds for $C/K$.
\end{proof}

\begin{remark}
In fact by Theorem \ref{2group} it suffices to assume only that $\sha(\Jac C/L)[p^{\infty}]$ is finite for all $p$ dividing the size of $\textup{Gal}(L/K) \subset \textup{Sp}_6(\F_2)$, i.e. for $p = 2, 3, 5$ and $7$, as the latter has size $1454120 = 8! \cdot 36$.
\end{remark}

\section{Explicit Methods for Computing Local Terms}

As previously noted, one expects to be able to evaluate a local formula at a given curve (and, in our case, choice of double cover). In this section we introduce methods to do this for the local formula derived in Section \ref{derivation}. First, we examine how the kernel/cokernel ratio can be understood in terms of more readily computable data (varying by place). Second, 
we give methods to compute this data explicitly (under mild assumptions on the curve $C$).

\subsection{Kernel/Cokernel Locally}

\begin{lemma}
\label{lockercoker}
Let $K$ be a number field with place $v$. Let $\phi \colon A \times B \to A'$ be a $2^g$-isogeny of abelian varieties (where $A \times B, A'$ have dimension $g$) over the local field $K_v$. Then
$$
\frac{|\ker \phi_v|}{|\coker \phi_v |} =
\begin{cases}
2^g & K_v \simeq \C \\
\frac{n_{A/\R} \, n_{B/\R}}{n_{A'/\R}} \, | \ker \phi|_{\R}^0 | & K_v \simeq \R \\
\frac{c_{A,v} \, c_{B,v} }{c_{A',v}} & \text{$K_v/\Q_p$ finite, $p$ odd} 
\end{cases}
$$
\end{lemma}

\begin{proof}
 Exactly as in \cite[Lem.~3.4]{dokchitser2020parity}. All the terms involved are multiplicative, i.e. $c_{A \times B,v} = c_{A,v} \, c_{B,v}$ and $n_{A \times B/K} = n_{A/K} \, n_{B/K}$.
\end{proof}

\begin{remark}[{cf. \cite[Lem.~3.4]{dokchitser2020parity}}]
A similar description can be given when $K_v$ is a finite extension of $\Q_2$, with
$$
\frac{|\ker \phi_v|}{|\coker \phi_v |} = \frac{c_{A,v} \, c_{B,v} }{c_{A',v}} \left| \frac{\phi^{*}\omega^o_{A',v}}{\omega^o_{A \times B , v}} \right|.
$$
where $\omega^o$ denotes the N\'{e}ron exterior form. This description will not be useful for us however, and we will appeal to Theorem \ref{2thm} when computing the local term at $2$-adic places.
\end{remark}

\subsection{Kernel of the Prym Isogeny}

The description of the local kernel/cokernel ratio in Lemma \ref{lockercoker} uses the kernel of the Prym isogeny at real places, and so a working description of it will be of benefit. This will also be useful at $2$-adic places. Here, then, we describe this kernel.

\begin{lemma}
\label{kernel}
    Let $C$ be a curve of genus $g$ with  Prym variety $\prym(D/C)$, associated $2$-torsion point $\epsilon$, and Prym isogeny $\phi$. Then
    $$
    \ker \phi = \left\{ (\alpha , \beta) \in \Jac C[2] \times \prym(D/C)[2] \, \mid \, \pi^{*}(\alpha) = \beta \right\}
    $$
\end{lemma}

\begin{proof}
Certainly $\ker \phi \subset \Jac C[2] \times \prym(D/C)[2]$. If $(\alpha, \beta) \in \ker \phi$, then $\pi^{*}(\alpha) = -\beta = \beta$. Each of the $2^{2(g-1)}$ elements of $\prym(D/C)[2]$ lies in the image of $\pi^{*}$ by Lemma \ref{2tors}. As $\ker \pi^{*} = \{ 0 , \epsilon \}$, there are then $2 \cdot 2^{2(g-1)} = | \ker \phi |$ pairs $( \alpha , \beta )$ with $\pi^{*}(\alpha) = \beta$. 
\end{proof}

Whilst this description is succinct, it does not tell us how to compute the kernel. To do this, it is necessary to recall the structure of $2$-torsion on both non-hyperelliptic curves of genus 3 and general hyperelliptic curves. 

\subsubsection{Two-torsion on Non-Hyperelliptic Curves of Genus 3} We summarise \cite[\S 6]{dolgachev}. We will consider non-hyperelliptic genus 3 curves as plane quartics. By a classic result, a plane quartic $C$ has 28 bitangents and any pair of bitangents specifies a $2$-torsion point of $\Jac C$. However ${28 \choose 2} = 378$ so this overcounts $\Jac C[2]$ by a factor of $6$, and indeed each $2$-torsion point is identified by $6$ distinct pairs of bitangents. Such a sextuplet of pairs is called a \textit{Steiner complex}, and two pairs of bitangents belong to the same Steiner complex if and only if the 8 points of tangency lie on a conic. We will specify a $2$-torsion point by giving one or more pairs of bitangents.

\subsubsection{Two-torsion on Hyperelliptic Curves} We summarise \cite[\S 5.2.2]{dolgachev}. Suppose $C \colon y^2 = f(x)$ is a hyperelliptic curve of genus $g$, where (without loss of generality), $\deg f = 2g + 2$. Consider the set of subsets of $B_g = \{1, 2, \ldots, 2g + 2 \}$ with even cardinality, modulo the relation $I \sim B_g \setminus I$, and equipped with the symmetric sum $I + J = I \cup J \setminus ( I \cap J )$. Denote the quotient $E_g$. Then there is an isomorphism $E_g \simeq \Jac C[2]$. Each element of $E_g$ is represented by some subset $I \subset B_g$ of even cardinality, with $B_g \setminus I$ belonging to the same class. Thus we can indicate particular $2$-torsion points by giving an (unordered) list of points on $C$ of the form $(\alpha, 0)$, where $f(\alpha) = 0$. In genus $2$ and $3$ we need only use lists of size $2$ or $4$.

\subsubsection{Two-Torsion on $\Jac C \times \prym(D/C)$}

When $\prym(D/C)$ is of the form $\Jac (y^2 = f(x)) \times \Jac (y^2 = g(x))$ then by abuse of notation, the points coming from the roots of $f$ (labelled $P_i$) and the roots of $g$ (labelled $P'_i$) will also be considered as points of $\Jac C = \Jac \left(y^2 = f(x)g(x)\right)$. Then to give $(\alpha, \beta) \in \Jac C [2] \times \prym(D/C) [2]$ we give $\alpha$ as a list of some $P_i$ and $P'_i$, and $\beta$ as two lists, the first consisting only of the $P_i$, the second only of the $P'_i$. If $\deg g = 2$ (so that $\Jac(y^2 = g(x))$ is trivial), we omit the second list. 

The case when $C$ is a non-hyperelliptic curve of genus $3$ will require a different labelling scheme. Recall that such curves admitting a Prym variety are of the form $Q_1(x,y,z) Q_3(x,y,z) - Q_2(x,y,z)^2 = 0$ by Table \ref{prymtable}. The Prym variety is then the Jacobian of the curve $y^2 = - \det(Q_1 + 2x Q_2 + x^2 Q_3)$, where the $Q_i$ are also considered as symmetric $3 \times 3$ matrices. The roots of $-\det(Q_1 + 2x Q_2 + x^2 Q_3)$ yield degenerate conics which are pairs of bitangents, and the six such pairs from all the roots form a Steiner complex. Hence in this case $(\alpha, \beta) \in \Jac C[2] \times \prym(D/C)[2]$ can be specified as follows: $\alpha$ as (up to) six pairs of bitangents, all belonging to the same Steiner complex, $\beta$ as any two of those six pairs (this giving two roots of $-\det(Q_1 + 2 x Q_2 + x^2 Q_3)$).

\subsubsection{Description of $(\pi^{*})^{-1}$} 

In light of Lemma \ref{kernel}, we must describe $(\pi^{*})^{-1}(\beta)$ for $\beta \in \prym(D/C)[2]$. There will be two such points in the pre-image, though it suffices to find only one, as if $\pi^{*}(\alpha) = \beta$, then $\pi^{*}(\alpha + \epsilon) = \beta$ also.

\begin{proposition}
\label{pull-back}
Let $C$ be a curve with $2$-torsion point $\epsilon$, and corresponding Prym variety  $\prym(D/C)$. For $\beta \in \prym(D/C)[2]$, $(\pi^{*})^{-1}(\beta)$ is as described in the following table.

\begin{table}[h!]
\centering
\begin{tabular}{|c|c|c|c|}
\hline
Case  & $C$                                                & $\beta$                                                                         & $(\pi^{*})^{-1}(\beta)$                                                                                                                \\ \hline
II    & $y^2 = f(x)g(x), \, \deg f = 4, \deg g = 2$        & $[P_i, P_j]$                                                                    & $[P_i, P_j]$, $[P_i, P_j] + \epsilon$                                                                                             \\ \hline
III.a & $y^2 = f(x)g(x), \, \deg f = 6, \deg g = 2$        & $[P_i, P_j]$                                                                    & $[P_i, P_j]$, $[P_i, P_j] + \epsilon$                                                                                             \\ \hline
III.b & $y^2 = f(x)g(x), \, \deg f = 4, \deg g = 4$        & $\left( [P_i, P_j], [P'_k, P'_l] \right)$                                       & \multirow{2}{*}{ \parbox{3cm}{\centering $[P_i, P_j, P'_k, P'_l] $, \\ $[P_i, P_j, P'_k, P'_l] + \varepsilon $ } }                        \\ \cline{1-3}
III.c & $y^2 = N_{K(\sqrt{d})[x]/K[x]}R(x), \, \deg R = 4$ & $\left( [P_i, P_j], [P'_k, P'_l] \right)$                                       &                                                                                                                                   \\ \hline
III.d & $Q_1(x,y,z)Q_3(x,y,z) - Q_2(x,y,z)^2 = 0$          & \parbox{3cm}{\centering $\left[ \{ b_1 , b_2 \} , \{ b_3 , b_4 \} \right] $, \\ $b_i$ bitangents } & \parbox{3cm}{\centering \footnotesize $\{ \{b_1, b_3\}, \{ b_2, b_4 \}, \ldots \}$, \\ $\{ \{b_1, b_4\}, \{ b_2, b_3 \}, \ldots \}$} \\ \hline
\end{tabular}
\end{table}

In Case III.c, the second factor of $\prym(D/C)$ is the Jacobian of the conjugate curve $y^2 = f^{\sigma}(x)$, so we can identify $P'_k = P^{\sigma}_k$. Note also that in Case III.d the pairs of bitangents $\{ b_1, b_3 \}$ and $\{ b_2, b_4\}$ (resp. $\{ b_1, b_4 \}$ and $\{ b_2, b_3 \}$) do indeed belong to the same Steiner complex, as the eight intersection points of these four bitangents all lie on a conic.

\end{proposition}

\begin{proof}
This can be done explicitly with divisors. We do this for Case II, and note that the other hyperelliptic cases are similar, though in Case III.d the map $\Jac F \hookrightarrow \Jac D$ is more subtle, and we treat it separately.

In Case II, $D$ is given by
$$
D \colon
\begin{cases}
y^2 = f(x) \\
z^2 = g(x)
\end{cases}
$$
in affine $3$-space, with projection map $\pi_1 \colon D \to F$, $\pi_1(x,y,z) = (x,y)$. The pull-back $\pi_1^{*}$ gives an isomorphism between $\Jac F$ and $\prym(D/C)$ \cite[Prop.~2.2]{BruinGenus3}, and we use this to move from divisors on $F$ to divisors on $D$.

Write $\gamma_i$ for the roots of $f$. Then $[P_i, P_j]$ on $\Jac F$ can be represented as the divisor $(\gamma_i, 0) - (\gamma_j, 0)$. Under the isomorphism $\pi_1^{*}$ this yields the divisor
$$
(\gamma_i, 0, \sqrt{g(\gamma_i)}) + (\gamma_i, 0, -\sqrt{g(\gamma_i)}) - (\gamma_j, 0, \sqrt{g(\gamma_j)}) - (\gamma_j, 0, -\sqrt{g(\gamma_j)})
$$
on $\Jac D$. We note simply that this agrees with
$$
\pi^{*}( (\gamma_i, 0) - (\gamma_j, 0) ) = (\gamma_i, 0, \sqrt{g(\gamma_i)}) + (\gamma_i, 0, -\sqrt{g(\gamma_j)}) - (\gamma_j, 0, \sqrt{g(\gamma_j)}) - (\gamma_j, 0, -\sqrt{g(\gamma_j)}).
$$

We now consider Case III.d. First we describe the map $\Jac F \hookrightarrow \Jac D$, summarising \cite[\S~4,5]{BruinGenus3}. For each $P = (x, y) \in F$, there is a quadric
$$
Q_1(u,v,w) + 2x \, Q_2(u,v,w) + x^2 \, Q_3(u,v,w) - (r + xs)^2
$$
in $u, v, w, r$ and $s$. The zero-set of this quadric contains two rulings of $2$-planes, one coming from $(x,y)$, the other from $(x,-y)$. Let $V^+$ be a plane from the $(x,y)$-ruling. Write $\mathcal{U}_P = D \cdot V^+$. A point of $\Jac F$ can be represented as a divisor of the form $P_1 + P_2 - \kappa_F$, then the image of $\mathcal{D}$ under $\Jac F \hookrightarrow \Jac D$ is $\mathcal{U}_{P_1} + \mathcal{U}_{P_2} - \pi^{*}(\kappa_C)$.

Let $(x_1,0) + (x_2,0) - \kappa_F$ represent a two-torsion point of $\Jac F$. The conic $Q_1 + 2x_1 Q_2 + x_1^2 Q_3$ gives a pair of bitangents $b_1, b_2$ with linear forms $l_1(u,v,w), l_2(u,v,w)$ respectively. Similarly $x_2$ gives bitangents (resp. linear forms) $b_3, b_4$ (resp. $l_3(u,v,w), l_4(u,v,w)$). We can write down one of the planes, $V$, contained in the zero-set of $l_1(u,v,w) l_2(u,v,w) - (r + x_1 s)^2 = 0$; it is the set of points $(\gamma_1, \gamma_2, \gamma_3, -x_1 t, t)$ such that $l_1(\gamma_1, \gamma_2, \gamma_3) = 0$. Let $P_i, Q_i$ be the two points of intersection of $b_i$ and $C$. Then $D \cdot V = \pi^{*}(P_1) + \pi^{*}(Q_1)$. So $\Jac F \hookrightarrow \Jac D$ sends $[(x_1,0), (x_2,0)]$ to (the class of) $\pi^{*}(P_1) + \pi^{*}(Q_1) + \pi^{*}(P_3) + \pi^{*}(Q_3) - \pi^{*}\kappa_C$. It suffices to observe that this is the pull-back of the divisor
$$
P_1 + Q_1 + P_3 + Q_3 - \kappa_C, 
$$
which comes from the bitangent pair $\{ b_1, b_3 \}$.
    
\end{proof}

\subsection{Local Terms at Particular Places}

\subsubsection{Archimedean Places}

The contribution from complex places is specified completely by Lemma \ref{lockercoker}, and so will not be discussed further. Consider, then, the case of real archimedean places. The following two theorems are used to compute the real term in the local formula. 

\begin{theorem}[{See \cite[Prop.~3.2.2,~3.3]{GrossandHarris}}]
\label{numbercomps}
Let $C$ be a genus $g$ curve over $\R$. Then
$$
n_{\Jac C / \R} = 
\begin{cases}
2^{n_{C/\R} - 1} &\text{if $n_{C/\R} > 0$} \\
2  &\text{if $n_{C/\R} = 0$ and $g$ is odd} \\
1 &\text{if $n_{C/\R} = 0$ and $g$ is even}
\end{cases}
$$
\end{theorem}

\begin{definition}
Let $X/\R$ be a curve and suppose $\textup{Comp}(X) = \{X_1, \ldots X_n \}$ is the set of real components of $X$, with $n > 0$. For a divisor $D = \sum n_i P_i \in \textup{Div}^0{X}$, let
$$
d_{X_j}(D) = \sum_{P_i \in X_j} n_i \mod 2.
$$
Let $d \colon \textup{Div}^{0}(X) \to (\Z/2\Z)^n$ be the function $d(D) = (d_{X_1}(D), \ldots, d_{X_n}(D))$. By \cite[Lem.~4.1]{GrossandHarris}, $d$ descends to a map $\Jac X \to (\Z/2\Z)^n$.
\end{definition}

\begin{corollary}[{See \cite[\S~4.1]{realabelianvarieties}, \cite[\S.~2]{realhypersurfaces}}]
\label{realcomps}
Let $X$ be as above with $X(\R) \neq \emptyset$. Then two points $P_1, P_2$ in $\Jac X(\R)$ belong to the same real component if and only if $d(P_1) = d(P_2)$. In particular, a divisor belongs to $\Jac X (\R)^0$ precisely when it has even intersection degree with all components of $X$.
\end{corollary}

\begin{remark}
For a genus $g$ curve $X$, $(\Jac X)(\R)^0$ is a $2^g$-dimensional real Lie manifold, and so has $2^g$ real $2$-torsion points. Simply counting the points is not enough here, though, as we need to identify the points explicitly. We also remark that when $X(\R) = \emptyset$, identifying which component a given divisor of $\Jac X(\R)$ belongs to is often more delicate, but we will not need this for the following examples. 
\end{remark}

\subsubsection{Non-Archimedean Places not Above $2$}

Evaluating the local formula at non-archimedean places not above $2$ reduces to calculating Tamagawa numbers by Lemma \ref{lockercoker}. In practice this can be done using SAGE \cite{sagemath} or \MAGMA \cite{magma}.

\subsubsection{Non-Archimedean Places Above $2$}

Analysing the local formula at places above $2$ in the same manner as those not above $2$ introduces additional terms which can be difficult to manage. It is preferable to forgo this entirely, and invoke the theorem below. There is a trade-off, though, in the control that is required over both the base curve and the Prym variety.

\begin{theorem}[{See \cite[Thm.~A.1]{dokchitser2020parity}}]
\label{2thm}
Let $C$ be a curve of genus $g = 2$ or $3$ over a finite extension $\mathcal{K}$ of $\Q_2$, with Prym variety $\prym(D/C)$, and let $A = \Jac C \times \prym(D/C)$. Suppose that $A$ has good ordinary reduction, and write $A_1(\bar{K})$ for the kernel of reduction. Then
$$
\frac{|\coker\phi|_K|}{| \ker\phi |_K |} = 2^{[ K \colon \Q_2 ] \dim_{\F_2} \left( A_1(\bar{K})[2] \cap A(\bar{K})[\phi] \right)}
$$
\end{theorem}

\section{Example Computation}

We now show how the local formula can be used by giving a worked example.

Consider $C \colon y^2 = f(x)g(x)$, where $f(x) = x^6 - 12x^5 + 48x^4 + 54x^3 + 60x^2 - 236x - 295$, $g(x) = (x+6)(x+2)$. We have Prym variety $\prym(D/C) = \Jac F$ where $F$ is the curve $y^2 = f(x)$, arising from the factorisation $f \cdot g$. The double cover $D$ has model $y^2 = x^{12} + 36x^{11} + 534x^{10} + 4094x^9 + 17667x^8 + 44018x^7 + 61093x^6 + 44018x^5 + 17667x^4 + 4094x^3 + 534x^2 + 36x + 1$. All curves have obvious rational points, and hence are nowhere deficient. By computing the discriminant of $C$, we find that the primes $p$ which potentially have non-trivial contribution to the local formula are $p = 2, 5, 7, 59, 653, 1201, 193793, 17283342701$ and $\infty$.

\subsection{$ p = \infty $}

By inspection, $n_{C/\R} = n_{D/\R} = 2$ and $n_{F/\R} = 1$, so by Theorem \ref{numbercomps}, $n_{\Jac C, \R} = n_{\Jac D, \R} = 2$ and $n_{\prym(D/C), \R} = 1$.

We now determine $| \ker \phi |_\R^0 |$. Let $\gamma_i$ be the roots of $f$, indexed so that $\gamma_{1,2}$ are real ($\gamma_1 < \gamma_2$), and $\gamma_{i} = \bar{\gamma}_{i+1}$ for $i = 3, 5$. Write $P_i = (\gamma_i, 0)$. We expect $\prym(D/C)(\R)^0$ to have four two-torsion points. By Corollary \ref{realcomps}, $[P_1, P_2], [P_3, P_4], [P_5, P_6]$ and $0$ are verified to be those four.
According to Lemma \ref{kernel}, we must determine the pre-images of these four points under $\pi^{*}$ and count which lie on $\Jac C(\R)^0$. For ease we write $P_7 = (-6, 0), P_8 = (-2, 0)$. Using Proposition \ref{pull-back}, $(\pi^{*})^{-1}[P_1, P_2] = \{ [P_1, P_2] , [P_1, P_2, P_7, P_8] \}$, $(\pi^{*})^{-1}[P_3, P_4] = \{ [P_3, P_4] , [P_3, P_4, P_7, P_8] \}$, $(\pi^{*})^{-1}[P_5, P_6] = \{ [P_5, P_6] , [P_5, P_6, P_7, P_8] \}$, $(\pi^{*})^{-1}(0) = \{ [P_7, P_8] , 0 \}$. By Corollary \ref{realcomps}, only the latter of each set lies on $\Jac C (\R)^0$. In particular, $| \ker \phi |_\R^0 | = 4$.

Then $\lambda_{C/\R, \phi_\R} = (-1)^{\ord_{2} \left( \frac{2}{2 \cdot 1 \cdot 4} \right)} = 1$.

\subsection{$p > 2$, finite} The Tamagawa numbers are found with SAGE, as in the table below

\begin{table}[h]
\centering
\begin{tabular}{c|ccccccc}
    & 5 & 7 & 59 & 653 & 1201 & 193793 & 17283342701 \\ \hline
$C$ & 1  & 1  & 1   & 1 & 2 & 2 & 1    \\
$F$ & 1  & 1  & 1   & 1 & 1 & 1 & 1   \\
$D$ & 1  & 1  & 1   & 1 & 1 & 1 & 1  
\end{tabular}
\end{table}

\subsection{$p = 2$}

Again we label $P_i = (\alpha_i, 0)$ for $i \in \{ 1, \ldots, 6\}$ with $f(\alpha_i) = 0$, and $P_7 = (-6, 0), P_8 = (-2, 0)$. $F$ has model $y^2 - (x^3 + 1)y = -3x^5 + 12x^4 + 13x^3 + 15x^2 - 59x - 74$, via $(x,y) \mapsto (x, \frac{y + x^3 + 1}{2})$, and reduces to the curve with LMFDB label 2.2.c\_d; in particular it has good ordinary reduction. Using \MAGMA, there are four $2$-torsion points in the kernel of reduction, from the points $[P_1, P_2], [P_3, P_4], [P_5, P_6]$ and $0$. To invoke Theorem \ref{2thm} we must consider the pre-image of these points under $\pi^{*}$ once more. We have $ (\pi^{*})^{-1}[P_1, P_2] = \{ [P_1, P_2] , [P_1, P_2, P_7, P_8] \}$, $(\pi^{*})^{-1}[P_3, P_4] = \{ [P_3, P_4] , [P_3, P_4, P_7, P_8] \}$, $(\pi^{*})^{-1}[P_5, P_6] = \{ [P_5, P_6] , [P_5, P_6, P_7, P_8] \}$, $(\pi^{*})^{-1}(0) = \{ [P_7, P_8] , 0 \}$. We determine which of these eight are in the kernel of reduction of $C$.

$C$ has model $y^2 - (x^4 + x)y = - x^7 - 9x^6 + 73x^5 + 267x^4 + 223x^3 - 366x^2 - 1298x - 885$ via $(x,y) \mapsto (x, \frac{y + x^4 + x}{2})$, and reduces to the curve with LMFDB label 3.2.b\_b\_d (with good ordinary reduction). \MAGMA gives that all are in the kernel of reduction. As both $C$ and $F$ have good ordinary reduction, we may invoke Theorem \ref{2thm} for
$$
\frac{|\coker\phi_{2}|}{| \ker\phi_{2} |} = 2^{\dim_{\F_2} 8} = 8,
$$
So $\lambda_{C/\Q_2,\phi_2} = (-1)^3 = -1$. 

Altogether, Theorem \ref{localformula} gives $(-1)^{\rk_2 \Jac C + \rk_2 \prym(D/C)} = -1$. However $F$ is semistable, so the $2$-parity conjecture is known for $\prym(D/C)/\Q$ \cite[Thm.~1.4]{dokchitser2020parity}, and computing root numbers for $\prym(D/C)$ we find that $(-1)^{\rk_2 \prym(D/C)} = - 1$. We thus isolate $\rk_2 \Jac C$ as even. Computing root numbers in SAGE, this agrees with $w_{\Jac C/K}$ (and hence Conjecture \ref{loceqroot}).

\begin{remark}
    In addition the the above example, we have sought to verify Conjecture \ref{loceqroot} for a number of curves (with non-trivial $2$-torsion over $\Q$) in a $2$-adic neighbourhood of $C$. We considered the 728 curves of the form
    $$
    y^2 = (f(x) + a_5 x^5 + a_4 x^4 + a_3 x^3 + a_2 x^2 + a_1 x + a_0) g(x)
    $$
    $a_i \in \{ 0, 8, 16 \}$ (excluding the example case where all $a_i$ are zero), with choice of double cover arising from this factorisation. In many instances the reduction type precluded computation of the $p$-adic contribution to the local formula. Nonetheless, the local formula could be evaluated successfully for 419 of them. In all such instances, we found Conjecture \ref{loceqroot} to hold. 
\end{remark}

\begin{remark}
We comment briefly on how the above example was found. In order to compute the local terms our main constraint was controlling both $C$ and $\prym(D/C)$ with Theorem \ref{2thm}. The recent development of \textit{cluster pictures} focused our search. It is known \cite[Prop.~8.4]{dokchitser2020parity} that hyperelliptic curves of genus $2$, have good ordinary reduction at $p = 2$ when they have cluster picture
\begin{center}
\clusterpicture
\Root[D] {} {first} {r1};
\Root[D] {} {r1} {r2};
\ClusterLD c1[][2] = (r1)(r2);
\Root[D] {} {c1} {r3};
\Root[D] {} {r3} {r4};
\ClusterLD c2[][2] = (r3)(r4);
\Root[D] {} {c2} {r5};
\Root[D] {} {r5} {r6};
\ClusterLD c3[][2] = (r5)(r6);
\ClusterLD c4[][0] = (c1)(c2)(c3);
\endclusterpicture
\end{center}
Hence we selected a polynomial $f(x)$ of degree 6 with this particular cluster picture, and then altered the second factor $g(x)$ until the curve $C \colon y^2 = f(x)g(x)$ has the required properties. In particular, when searching for a hyperelliptic curve of genus 3 with good ordinary reduction at $p = 2$, it was natural to try instances of $f$ and $g$ such that $C$ had cluster picture
\begin{center}
\clusterpicture
\Root[A] {} {first} {r1};
\Root[A] {} {r1} {r2};
\ClusterLD c1[][2] = (r1)(r2);
\Root[A] {} {c1} {r3};
\Root[A] {} {r3} {r4};
\ClusterLD c2[][2] = (r3)(r4);
\Root[A] {} {c2} {r5};
\Root[A] {} {r5} {r6};
\ClusterLD c3[][2] = (r5)(r6);
\Root[C] {} {c3} {r7};
\Root[C] {} {r7} {r8};
\ClusterLD c5[][2] = (r7)(r8);
\ClusterLD c6[][0] = (c1)(c2)(c3)(c5);
\endclusterpicture
\end{center}
(where \raisebox{0.4em}{\clusterpicture \Root[A] {}{first}{r1}; \endclusterpicture} indicates a root of $f$, \raisebox{0.4em}{\clusterpicture \Root[C] {}{first}{r1}; \endclusterpicture} a root of $g$). It is fortunate that the reduction of $\prym(D/C)$ can be seen so straightforwardly in the reduction of $C$ using cluster pictures. No such analogue is currently known for non-hyperelliptic genus $3$ curves, and we note that it seems particularly difficult to find examples of case III.d which are amenable to explicit computation. We did not find examples over $\Q$ where both $C$ and the Prym variety $\prym(D/C)$ had good reduction at the prime $2$ (let alone good ordinary reduction). However it is not clear to the author that no such example can exist.
\end{remark}

\bibliography{2selmerbib}

\begin{thebibliography}{10}

\bibitem{bhargavagross}
M.~Bhargava and B.~H. Gross.
\newblock The average size of the 2-{S}elmer group of {J}acobians of
  hyperelliptic curves having a rational {W}eierstrass point.
\newblock In {\em Automorphic representations and {$L$}-functions}, volume~22
  of {\em Tata Inst. Fundam. Res. Stud. Math.}, pages 23--91. Tata Inst. Fund.
  Res., Mumbai, 2013.

\bibitem{bhargavashankar}
M.~Bhargava and A.~Shankar.
\newblock Binary quartic forms having bounded invariants, and the boundedness
  of the average rank of elliptic curves.
\newblock {\em Ann. of Math. (2)}, 181(1):191--242, 2015.

\bibitem{magma}
Wieb Bosma, John Cannon, and Catherine Playoust.
\newblock The {M}agma algebra system. {I}. {T}he user language.
\newblock volume~24, pages 235--265. 1997.
\newblock Computational algebra and number theory (London, 1993).

\bibitem{BruinGenus3}
N.~Bruin.
\newblock The arithmetic of {P}rym varieties in genus 3.
\newblock {\em Compos. Math.}, 144(2):317--338, 2008.

\bibitem{realabelianvarieties}
C.~Ciliberto and C.~Pedrini.
\newblock Real abelian varieties and real algebraic curves.
\newblock In {\em Lectures in real geometry ({M}adrid, 1994)}, volume~23 of
  {\em De Gruyter Exp. Math.}, pages 167--256. de Gruyter, Berlin, 1996.

\bibitem{dixmier}
J.~Dixmier.
\newblock On the projective invariants of quartic plane curves.
\newblock {\em Adv. in Math.}, 64(3):279--304, 1987.

\bibitem{notesonpc}
T.~Dokchitser.
\newblock Notes on the parity conjecture.
\newblock In {\em Elliptic curves, {H}ilbert modular forms and {G}alois
  deformations}, Adv. Courses Math. CRM Barcelona, pages 201--249.
  Birkh\"{a}user/Springer, Basel, 2013.

\bibitem{cyclicpc}
T.~Dokchitser and V.~Dokchitser.
\newblock Parity of ranks for elliptic curves with a cyclic isogeny.
\newblock {\em J. Number Theory}, 128(3):662--679, 2008.

\bibitem{BSDModSquares}
T.~Dokchitser and V.~Dokchitser.
\newblock On the {B}irch-{S}winnerton-{D}yer quotients modulo squares.
\newblock {\em Annals of Mathematics}, 172(1):567–596, Jun 2010.

\bibitem{rootnumbersnonab}
V.~Dokchitser.
\newblock Root numbers of non-abelian twists of elliptic curves.
\newblock {\em Proc. London Math. Soc. (3)}, 91(2):300--324, 2005.
\newblock With an appendix by Tom Fisher.

\bibitem{dokchitser2020parity}
V.~Dokchitser and C.~Maistret.
\newblock Parity conjecture for abelian surfaces, 2020.

\bibitem{dolgachev}
I.~V. Dolgachev.
\newblock {\em Classical algebraic geometry}.
\newblock Cambridge University Press, Cambridge, 2012.
\newblock A modern view.

\bibitem{GrossandHarris}
B.~H. Gross and J.~Harris.
\newblock Real algebraic curves.
\newblock {\em Ann. Sci. \'{E}cole Norm. Sup. (4)}, 14(2):157--182, 1981.

\bibitem{realhypersurfaces}
J.~Huisman.
\newblock On the number of real hypersurfaces hypertangent to a given real
  space curve.
\newblock {\em Illinois J. Math.}, 46(1):145--153, 2002.

\bibitem{kramertunnell}
K.~Kramer and J.~Tunnell.
\newblock Elliptic curves and local {$\varepsilon $}-factors.
\newblock {\em Compositio Math.}, 46(3):307--352, 1982.

\bibitem{mazurrubin}
B.~Mazur and K.~Rubin.
\newblock Ranks of twists of elliptic curves and {H}ilbert's tenth problem.
\newblock {\em Invent. Math.}, 181(3):541--575, 2010.

\bibitem{MumfordPryms}
D.~Mumford.
\newblock Prym varieties. {I}.
\newblock In {\em Contributions to analysis (a collection of papers dedicated
  to {L}ipman {B}ers)}, pages 325--350. 1974.

\bibitem{ohno}
T.~Ohno.
\newblock The graded ring of invariants of ternary quartics i.
\newblock {\em unpublished}, 2005.

\bibitem{poonenandstoll}
B.~Poonen and M.~Stoll.
\newblock The {C}assels-{T}ate pairing on polarized abelian varieties.
\newblock {\em Ann. of Math. (2)}, 150(3):1109--1149, 1999.

\bibitem{prymtheory}
V.~V. Shokurov.
\newblock Prym varieties: theory and applications.
\newblock {\em Izv. Akad. Nauk SSSR Ser. Mat.}, 47(4):785--855, 1983.

\bibitem{smith2inftyselmer}
A.~Smith.
\newblock $2^\infty$-selmer groups, $2^\infty$-class groups, and goldfeld's
  conjecture, 2017.

\bibitem{sagemath}
{The Sage Developers}.
\newblock {\em {S}ageMath, the {S}age {M}athematics {S}oftware {S}ystem
  ({V}ersion 9.2)}, 2021.
\newblock {\tt https://www.sagemath.org}.

\end{thebibliography}
\bibliographystyle{plain}

\end{document}